\documentclass[11pt, a4paper]{article}
\usepackage{amsmath, amsthm, amssymb, url}
\usepackage[margin=2cm]{geometry}
\usepackage{multirow}
\usepackage{tikz, tkz-graph, mathrsfs, enumerate}
\usepackage{authblk}

\usetikzlibrary{arrows}
\newcommand{\vertex}[3]{\node [vertex] (#1) at (#2, #3 * 1.7) {};}

\newcommand{\arc}[2]{{\draw[-latex] (#1) edge (#2);}}

\newcommand{\Tran}{\mathrm{Tran}}
\newcommand{\Sym}{\mathrm{Sym}}

\newcommand{\id}{\mathrm{id}}
\newcommand{\CA}{\mathrm{CA}}
\newcommand{\ICA}{\mathrm{ICA}}
\newcommand{\Rank}{\mathrm{Rank}}
\newcommand{\Sub}{\mathrm{Sub}}
\newcommand{\ac}{\mathrm{ac}}
\newcommand{\Fix}{\mathrm{Fix}}
\newcommand{\Aut}{\mathrm{Aut}}

\theoremstyle{plain}

\newtheorem{corollary}{Corollary}
\newtheorem{lemma}{Lemma}

\newtheorem{theorem}{Theorem}
\newtheorem*{claim}{Claim}

\theoremstyle{definition}
\newtheorem{definition}{Definition}
\newtheorem{problem}{Problem}
\newtheorem{example}{Example}
\newtheorem{remark}{Remark}

\begin{document}

\title{Cellular Automata and Finite Groups}
\author{Alonso Castillo-Ramirez and Maximilien Gadouleau}

\newcommand{\Addresses}{{
  \bigskip
  \footnotesize

A. Castillo-Ramirez (Corresponding author), \textsc{Universidad de Guadalajara, CUCEI, Departamento de Matem\'aticas, Guadalajara, M\'exico.} \par \nopagebreak
Email: \texttt{alonso.castillor@academicos.udg.mx}

\medskip

M. Gadouleau, \textsc{School of Engineering and Computing Sciences, Durham University, South Road, Durham DH1 3LE, U.K.} \par \nopagebreak
 Email: \texttt{m.r.gadouleau@durham.ac.uk}

}}

\maketitle

\begin{abstract}
For a finite group $G$ and a finite set $A$, we study various algebraic aspects of cellular automata over the configuration space $A^G$. In this situation, the set $\CA(G;A)$ of all cellular automata over $A^G$ is a finite monoid whose basic algebraic properties had remained unknown. First, we investigate the structure of the group of units $\ICA(G;A)$ of $\CA(G;A)$. We obtain a decomposition of $\ICA(G;A)$ into a direct product of wreath products of groups that depends on the numbers $\alpha_{[H]}$ of periodic configurations for conjugacy classes $[H]$ of subgroups of $G$. We show how the numbers $\alpha_{[H]}$ may be computed using the M\"obius function of the subgroup lattice of $G$, and we use this to improve the lower bound recently found by Gao, Jackson and Seward on the number of aperiodic configurations of $A^G$. Furthermore, we study generating sets of $\CA(G;A)$; in particular, we prove that $\CA(G;A)$ cannot be generated by cellular automata with small memory set, and, when all subgroups of $G$ are normal, we determine the relative rank of $\ICA(G;A)$ on $\CA(G;A)$, i.e. the minimal size of a set $V \subseteq \CA(G;A)$ such that $\CA(G;A) = \langle \ICA(G;A) \cup V \rangle$. \\

\end{abstract}

\section{Introduction}\label{intro}

Cellular automata (CA), introduced by John von Neumann and Stanislaw Ulam as an attempt to design self-reproducing systems, are models of computation with important applications to computer science, physics, and theoretical biology. In recent years, the mathematical theory of CA has been greatly enriched by its connections to group theory and topology (e.g., see \cite{CSC10} and references therein). The goal of this paper is to embark in the new task of exploring CA from the point of view of finite group theory and finite semigroup theory.

First of all, we review the broad definition of CA that appears in \cite[Sec.~1.4]{CSC10}. Let $G$ be a group and $A$ a set. Denote by $A^G$ the \emph{configuration space}, i.e. the set of all functions of the form $x:G \to A$. For each $g \in G$, let $R_g : G \to G$ be the right multiplication function, i.e. $(h)R_g := hg$, for any $h \in G$. We emphasise that we apply functions on the right, while in \cite{CSC10} functions are applied on the left.   

\begin{definition} \label{def:ca}
Let $G$ be a group and $A$ a set. A \emph{cellular automaton} over $A^G$ is a transformation $\tau : A^G \to A^G$ such that there is a finite subset $S \subseteq G$, called a \emph{memory set} of $\tau$, and a \emph{local function} $\mu : A^S \to A$ satisfying
\[ (g)(x)\tau = (( R_g \circ x  )\vert_{S}) \mu, \ \forall x \in A^G, g \in G,  \]
where $( R_g \circ x  )\vert_{S}$ denotes the restriction to $S$ of the configuration $ R_g \circ x : G \to A$.
\end{definition} 

Most of the classical literature on CA focuses on the case when $G=\mathbb{Z}^d$, for $d\geq1$, and $A$ is a finite set (e.g. see survey \cite{Ka05}).

A \emph{semigroup} is a set $M$ equipped with an associative binary operation. If there exists an element $\id \in M$ such that $\id \cdot m = m \cdot \id = m$, for all $m \in M$, the semigroup $M$ is called a \emph{monoid} and $\id$ an \emph{identity} of $M$. Clearly, the identity of a monoid is always unique. The \emph{group of units} of $M$ is the set of all \emph{invertible elements} of $M$ (i.e. elements $a \in M$ such that there is $a^{-1} \in M$ with $a \cdot a^{-1} = a^{-1} \cdot a = \id$). 

Let $\CA(G;A)$ be the set of all cellular automata over $A^G$; by \cite[Corollary 1.4.11]{CSC10}, this set equipped with the composition of functions is a monoid. Although results on monoids of CA have appeared in the literature before (see \cite{CRG16a,H12,S15}), the algebraic structure of $\CA(G;A)$ remains basically unknown. In particular, the study of $\CA(G;A)$, when $G$ and $A$ are both finite, has been generally disregarded (except for the case when $G=\mathbb{Z}_n$, which is the study of one-dimensional CA on periodic points). It is clear that many of the classical questions on CA are trivially answered when $G$ is finite (e.g. the Garden of Eden theorems become trivial), but, on the other hand, several new questions, typical of finite semigroup theory, arise in this setting.

In this paper, we study various algebraic properties of $\CA(G;A)$ when $G$ and $A$ are both finite. First, in Section \ref{basic}, we introduce notation and review some basic results. In Section \ref{structure}, we study the group of units $\ICA(G;A)$ of $\CA(G;A)$, i.e. the group of all invertible (also known as reversible) CA over $A^G$. We obtain an explicit decomposition of $\ICA(G;A)$ into a direct product of wreath products of groups that depends on the numbers $\alpha_{[H]}$ of periodic configurations for conjugacy classes $[H]$ of subgroups $H$ of $G$. 

In Section \ref{aperiodic}, we show how the numbers $\alpha_{[H]}$ may be computed using the M\"obius function of the subgroup lattice of $G$, and we give some explicit formulae for special cases. Furthermore, we make a large improvement on the lower bound recently found by Gao, Jackson and Seward \cite{GJS16} on the number of aperiodic configurations of $A^G$.  

Finally, in Section \ref{generating}, we study generating sets of $\CA(G;A)$. A set $T$ of CA is called a \emph{generating set} of $\CA(G;A)$ if every CA over $A^G$ is expressible as a word in the elements of $T$. We prove that $\CA(G;A)$ cannot be generated by CA with small memory sets: every generating set $T$ of $\CA(G;A)$ must contain a cellular automaton with minimal memory set equal to $G$ itself. This result provides a striking contrast with CA over infinite groups because, in such cases, the memory set of any cellular automaton may never be equal to the whole group (as memory sets are finite by definition). Finally, when $G$ is finite abelian, we find the smallest size of a set $V \subseteq \CA(G;A)$ such that $\ICA(G;A) \cup V$ generates $\CA(G;A)$; this number is known in semigroup theory as the \emph{relative rank} of $\ICA(G;A)$ in $\CA(G;A)$, and it turns out to be related with the number of edges of the subgroup lattice of $G$.      

The present paper is an extended version of \cite{CRG16b}. In this version, we added preliminary material in order to make the paper self-contained, we improved the exposition, we generalised several results (e.g. Corollary \ref{cor-conjugate}, Lemma \ref{Dedekind}, and Theorem \ref{th:relative rank}), and we added the completely new Section \ref{aperiodic}.


\section{Basic Results} \label{basic}

For any set $X$, a \emph{transformation} of $X$ is a function of the form $\tau : X \to X$. Let $\Tran(X)$ and $\Sym(X)$ be the sets of all transformations and bijective transformations of $X$, respectively. Equipped with the composition of transformations, $\Tran(X)$ is known as the \emph{full transformation monoid} on $X$, while $\Sym(X)$ is the \emph{symmetric group} on $X$. When $X$ is finite and $\vert X \vert = q$, we write $\Tran_q$ and $\Sym_q$ instead of $\Tran(X)$ and $\Sym(X)$, respectively. A \emph{finite transformation monoid} is simply a submonoid of $\Tran_q$, for some $q$. This type of monoids has been extensively studied (e.g. see \cite{GM09} and references therein), and it should be noted its close relation to finite-state machines. 

Recall that the \emph{order} of a group $G$ is simply the cardinality of $G$ as a set. For the rest of the paper, let $G$ be a finite group of order $n$ and $A$ a finite set of size $q$. By Definition \ref{def:ca}, it is clear that $\CA(G;A) \leq \Tran(A^G)$ (where we use the symbol ``$\leq$'' for the submonoid relation). We may always assume that $\tau \in \CA(G;A)$ has (not necessarily minimal) memory set $S = G$, so $\tau$ is completely determined by its local function $\mu: A^G \to A$. Hence, $\vert \CA(G ; A) \vert = q^{q^n}$. 

If $n=1$, then $\CA(G;A) = \Tran(A)$, while, if $q \leq 1$, then $\CA(G;A)$ is the trivial monoid with one element; henceforth, we assume $n \geq 2$ and $q \geq 2$. Without loss of generality, we identify $A$ with the set $\{0, 1, \dots, q-1 \}$ and we denote the identity element of $G$ by $e$.

A \emph{group action} of $G$ on a set $X$ is a function $\cdot : X \times G \to X$ such that $(x \cdot g) \cdot h = x \cdot gh$ and $x \cdot e = x$ for all $x \in X$, $g,h \in G$ (where we denote the image of a pair $(x,g)$ by $x \cdot g$). A group $G$ acts on the configuration space $A^G$ as follows: for each $g \in G$ and $x \in A^G$, the configuration $x \cdot g \in A^G$ is defined by 
\begin{equation} \label{action}
(h)x \cdot g = (hg^{-1})x, \quad \forall h \in G. 
\end{equation}
This indeed defines a group action because:
\begin{description}
\item[(i)] For all $h \in G$, $(h)x \cdot e = (h)x$, so $x \cdot e = x$. 
\item[(ii)] For all $h, g_1, g_2 \in G$, 
\[ (h) (x \cdot g_1) \cdot g_2 = (h g_2^{-1}) x \cdot g_1 = (h g_2^{-1}g_1^{-1})x = (h (g_1g_2)^{-1} ) x = (h) x \cdot g_1g_2,  \]
so $ (x \cdot g_1) \cdot g_2  = x \cdot g_1 g_2$.
\end{description}
Note that in equation (\ref{action}), $h$ has to be multiplied by the inverse of $g$ and not by $g$ itself, as property (ii) may not hold in the latter case when $G$ is non-abelian.

\begin{definition}
A transformation $\tau : A^G \to A^G$ is \emph{$G$-equivariant} if, for all $x \in A^G$, $g \in G$,
\[ (x \cdot g) \tau =  (x) \tau \cdot g .\]
\end{definition} 

\begin{theorem} \label{AG-finite}
Let $G$ be a finite group and $A$ a finite set. Then, 
\[ \CA(G;A) = \{ \tau \in \Tran(A^G) : \tau \text{ is $G$-equivariant} \}. \]
\end{theorem}
\begin{proof}
By Curtis-Hedlund Theorem (see \cite[Theorem 1.8.1]{CSC10}), a transformation $\tau \in \Tran(A^G)$ is a cellular automaton if and only if $\tau$ is $G$-equivariant and continuous in the prodiscrete topology of $A^G$ (i.e. the product topology of the discrete topology). However, as both $G$ and $A$ are finite, every transformation in $\Tran(A^G)$ is continuous, and the result follows. 
\end{proof}

In other words, the previous result means that $\CA(G;A)$ is the endomorphism monoid of the $G$-set $A^G$. This result allows us to study $\CA(G;A)$ form a pure algebraic perspective.  

We review a few further basic concepts on group actions (see \cite[Ch. 1]{DM96}). For $x \in A^G$, denote by $G_x$ the \emph{stabiliser} of $x$ in $G$:
\[ G_x : = \{g \in G : x \cdot g = x \}.\]

\begin{remark} \label{rk:subgroups}
For any subgroup $H \leq G$ there exists $x \in A^G$ such that $G_x = H$; namely, we may define $x : G \to A$ by
\[ (g)x := \begin{cases}
1 & \text{if } g \in H, \\
0 & \text{otherwise},
\end{cases} \quad \forall g \in G. \]
\end{remark}

For any $x \in A^G$, denote by $xG$ the \emph{$G$-orbit} of $x$ on $A^G$:
\[ xG := \{ x \cdot g : g \in G  \}. \]
Let $\mathcal{O}(G;A)$ be the set of all $G$-orbits on $A^G$:
\[ \mathcal{O}(G;A) := \{ xG : x \in A^G\}.\] 
It turns out that $\mathcal{O}(G;A)$ forms a partition of $A^G$. The following result is known as the Orbit-Stabiliser Theorem (see \cite[Theorem 1.4A.]{DM96}).

\begin{theorem} \label{orbit-stabiliser}
Let $G$ be a finite group and $A$ a finite set. For any $x \in A^G$,
\[ \vert x G \vert = \frac{\vert G \vert}{\vert G_x \vert}.\]
Moreover, if $x = y \cdot g$ for some $x,y \in A^G$, $g \in G$, then $G_x = g^{-1} G_y g$.
\end{theorem}

 In general, when $X$ is a set and $\mathcal{P}$ is a partition of $X$, we say that a transformation monoid $M \leq \Tran(X)$ \emph{preserves the partition} if, for any $P \in \mathcal{P}$ and $\tau \in M$ there is $Q \in \mathcal{P}$ such that $(P)\tau \subseteq Q$.

\begin{lemma} \label{preserve}
For any $x \in A^G$ and $\tau \in \CA(G;A)$,
\[ (xG) \tau = (x)\tau G. \]
In particular, $\CA(G;A)$ preserves the partition $\mathcal{O}(G;A)$ of $A^G$.
\end{lemma}
\begin{proof}
The result follows by the $G$-equivariance of $\tau \in \CA(G;A)$.   
\end{proof}

 A configuration $x \in A^G$ is called \emph{constant} if $(g)x = k \in A$, for all $g \in G$. In such case, we usually denote $x$ by $\mathbf{k} \in A^G$. 

\begin{lemma} \label{constant-config}
 Let $\tau \in \CA(G;A)$ and let $\mathbf{k} \in A^G$ be a constant configuration. Then, $(\mathbf{k}) \tau \in A^G$ is a constant configuration.
\end{lemma}
\begin{proof}
Observe that $x \in A^G$ is constant if and only if  $x \cdot g = x$, for all $g \in G$. By $G$-equivariance,
\[ (\mathbf{k}) \tau =  (\mathbf{k} \cdot g) \tau =  (\mathbf{k}) \tau \cdot g,  \quad \forall g \in G. \]
Hence, $(\mathbf{k})\tau$ is constant.  
\end{proof}

A \emph{subshift} of $A^G$ is a subset $X \subseteq A^G$ that is \emph{$G$-invariant}, i.e. for all $x \in X$, $g \in G$, we have $x \cdot g \in X$, and closed in the prodiscrete topology of $A^G$. As $G$ and $A$ are finite, the subshifts of $A^G$ are simply unions of $G$-orbits in $A^G$. 

The actions of $G$ on two sets $X$ and $Y$ are \emph{equivalent} if there is a bijection $\lambda : X \to Y$ such that, for all $x \in X, g \in G$, we have $(x \cdot g )\lambda = (x)\lambda \cdot g$.  

Two subgroups $H_1$ and $H_2$ of $G$ are \emph{conjugate} in $G$ if there exists $g \in G$ such that $g^{-1} H_1 g = H_2$. This defines an equivalence relation on the subgroups of $G$. Denote by $[H]$ the conjugacy class of $H \leq G$. If $y$ and $z$ are two configurations in the same $G$-orbit, then by Theorem \ref{orbit-stabiliser} we have $[G_y] = [G_z]$.

We use the cyclic notation for the permutations of $\Sym(A^G)$. If $B \subseteq A^G$ and $a \in A^G$, we define the idempotent transformation $(B \to a) \in \Tran(A^G)$ by
\[ (x)(B \to a) := \begin{cases}
a & \text{ if } x \in B, \\
x & \text{ otherwise},  
 \end{cases} \quad \forall x \in A^G. \]
When $B=\{ b\}$ is a singleton, we write $(b \to a)$ instead of $(\{ b\} \to a)$.


\section{The Structure of $\ICA(G;A)$} \label{structure}

Denote by $\ICA(G;A)$ the group of all invertible cellular automata:
\[ \ICA(G;A) := \{ \tau \in \CA(G;A) : \exists \tau^{-1} \in \CA(G;A) \text{ such that } \tau\tau^{-1} = \tau^{-1} \tau = \id \}. \]

As the inverse of a bijective $G$-equivariant map is also $G$-equivariant, it follows by Theorem \ref{AG-finite} that $\ICA(G;A) = \CA(G;A) \cap \Sym(A^G)$.

The following is an essential result for our description of the structure of the group of invertible cellular automata.

\begin{lemma}\label{le-ICA}
Let $G$ be a finite group of order $n \geq 2$ and $A$ a finite set of size $q \geq 2$. Let $x,y \in A^G$ be such that $xG \neq yG$. Then, there exists $\tau \in \ICA(G;A)$ such that $(x)\tau = y$ if and only if $G_x =  G_y$.
\end{lemma}
\begin{proof}
Suppose first that there is $\tau \in \ICA(G;A)$ such that $(x)\tau = y$. Let $g \in G_x$. Then, $y \cdot g = (x) \tau \cdot g = (x \cdot g) \tau = (x) \tau = y$, so $g \in G_y$. This shows that $G_x \leq G_y$. Now, let $h \in G_y$. Then, $x \cdot h = (y) \tau^{-1} \cdot h = (y \cdot h) \tau^{-1} = (y) \tau^{-1} = x$, where $\tau^{-1} \in \ICA(G;A)$ is the inverse of $\tau$, so $h \in G_x$. Therefore, $G_x = G_y$.

Suppose now that $G_x = G_y$. We define a map $\tau : A^G \to A^G$ as follows:
\[ (z)\tau := \begin{cases}
y \cdot g & \text{if } z = x \cdot g, \\
x \cdot g & \text{if } z = y \cdot g, \\
z & \text{otherwise},
\end{cases} \quad \quad \forall z \in A^G. \]   
We check that $\tau$ is well-defined:
\[ x \cdot g = x \cdot h \ \Leftrightarrow \ g h^{-1} \in G_x = G_y \ \Leftrightarrow \ y \cdot g = y \cdot h; \]
therefore, every element of $A^G$ has a unique image under $\tau$. Clearly, $\tau$ is $G$-equivariant and invertible (in fact, $\tau = \tau^{-1}$). Hence $\tau \in \ICA(G;A)$, and it satisfies $(x)\tau = y$.  
\end{proof}

\begin{corollary}\label{conjugate}
Under the assumptions of Lemma \ref{conjugate}, there exists $\tau \in \ICA(G;A)$ such that $(xG)\tau = yG$ if and only if $[G_x] = [ G_y]$
\end{corollary}
\begin{proof}
Suppose that $(xG)\tau = yG$. Then, $(x)\tau = y \cdot g$, for some $g \in G$. By Lemma \ref{conjugate}, $G_x = G_{y \cdot g}$. However, note that $G_{y \cdot g} = g^{-1} G_y g$, so $[G_x] = [G_y]$. Conversely, if $[G_x] = [G_y]$, then $G_x = g^{-1} G_y g = G_{y \cdot g}$, for some $g \in G$. By Lemma \ref{conjugate}, there exists $\tau \in \ICA(G;A)$ such that $(x)\tau = y \cdot g$, and by Lemma \ref{preserve}, $(xG)\tau = yG$.
\end{proof}

A subgroup $H \leq G$ is \emph{normal} if $[H] = \{ H \}$ (i.e. $g^{-1}H g = H$ for all $g \in G$).

\begin{corollary}\label{cor-conjugate}
Suppose that $G$ is a finite group whose subgroups are all normal. For any $x,y \in A^G$, there exists $\tau \in \ICA(G;A)$ such that $(xG)\tau = yG$ if and only if $G_x = G_y$.
\end{corollary}

Groups whose subgroups are all normal are called \emph{Dedekind groups}. Clearly, abelian groups are always Dedekind. The finite non-abelian Dedekind groups (also known as finite \emph{Hamiltonian groups}) were classified by Richard Dedekind in \cite{D97} and have the form $Q_8 \times (\mathbb{Z}_2)^n \times H$, where $Q_8$ is the quaternion group (i.e. the group of units of the quaternions $\mathbb{H}$), $n \geq 0$, and $H$ is a finite abelian group of odd order. Several of our stronger results on $\CA(G;A)$ will hold when $G$ is a finite Dedekind group.

For any integer $\alpha \geq 2$ and any group $C$, the \emph{wreath product} of $C$ by $\Sym_\alpha$ is the set
\[ C \wr \Sym_{\alpha} := \{ (v; \phi) : v \in C ^\alpha, \phi \in \Sym_\alpha \} \]
equipped with the operation
\[ (v;\phi) \cdot (w; \psi) = ( v w^{\phi}; \phi \psi), \text{ for any } v,w \in C^\alpha, \phi, \psi \in \Sym_\alpha, \]
where $\phi$ acts on $w$ by permuting its coordinates:
\[  w^\phi = (w_1, w_2, \dots, w_\alpha)^\phi := (w_{(1)\phi}, w_{(2)\phi}, \dots, w_{(\alpha)\phi}). \]
In fact, as may be seen from the above definitions, $C \wr \Sym_{\alpha}$ is equal to the external semidirect product $C^{\alpha} \rtimes_{\varphi} \Sym_{\alpha}$, where $\varphi : \Sym_\alpha \to \Aut(C^{\alpha})$ is given by $(w)(\phi)\varphi := w^\phi$, for all $w \in C^\alpha$, $\phi \in \Sym_\alpha$. For a more detailed description of the wreath product of groups, see \cite[Sec. 2.6]{DM96}.

Let $O \in \mathcal{O}(G;A)$ be a $G$-orbit. If $G_{(O)}$ is the pointwise stabiliser of $O$, i.e. $G_{(O)} := \bigcap_{x \in O} G_x$, the group $G^O := G / G_{(O)}$ is isomorphic to a subgroup of $\Sym(O)$ (as the homomorphism $\rho : G \to \Sym(O)$ given by $(x)(g)\rho = x \cdot g$, for all $x \in O, g \in G$, has kernel $G_{(O)}$;  see \cite[p. 17]{DM96} for more details). Abusing the notation, we also write $G^O$ for the isomorphic copy of $G^O$ inside $\Sym(O)$. Define the group
\begin{equation} \label{centraliser}
\ICA(O) := \{ \tau \in \Sym(O) : \tau \text{ is $G$-equivariant} \}.  
\end{equation}
Note that $\ICA(O)$ is isomorphic to the centraliser of $G^O$ in $\Sym(O)$: 
\[ \ICA(O) \cong C_{\Sym(O)}(G^O). \]

Let $H$ be a subgroup of $G$ and $[H]$ its conjugacy class. Define 
\[ B_{[H]}(G;A) := \{ x \in A^G : [G_x] = [H] \}. \]
Note that $B_{[H]}(G;A)$ is a subshift of $A^G$ (i.e. a union of $G$-orbits) and, by Theorem \ref{orbit-stabiliser}, all the $G$-orbits contained in $B_{[H]}(G;A)$ have equal sizes. Define 
\[ \alpha_{[H]} (G; A) := \left\vert \left\{ O \in \mathcal{O}(G,A) : O \subseteq B_{[H]}(G;A)   \right\} \right\vert.  \]
If $r$ is the number of different conjugacy classes of subgroups of $G$, observe that 
\[ \mathcal{B} := \{ B_{[H]}(G;A) : H \leq G \}\]
is a partition of $A^G$ into $r$ blocks. When $G$ and $A$ are clear from the context, we write simply $B_{[H]}$ and $\alpha_{[H]}$ instead of $B_{[H]}(G;A)$ and $\alpha_{[H]}(G;A)$, respectively. 

\begin{remark}
For any $G$ and $A$, we have $B_{[G]}(G;A) = \{ x \in A^G : x \text{ is constant} \}$ and $\alpha_{[G]} (G; A) = \vert A \vert$.
\end{remark}

\begin{example}
Let $G \cong \mathbb{Z}_n = \{0,1,\dots,n-1 \}$ be a cyclic group of order $n \geq 2$ and let $A$ be a finite set of size $q \geq 2$. Any configuration $x : \mathbb{Z}_n \to A$ may be represented by a $n$-tuple $(x_1, x_2, \dots, x_n)$ such that $x_i := (i-1)x$. The action of $\mathbb{Z}_n$ on $A^G$ correspond to cyclic shifts of the $n$-tuples; for example,
\[ (x_1, x_2, \dots, x_n) \cdot 1 = (x_n, x_1, x_2, \dots, x_{n-1}).  \]
As $\mathbb{Z}_n$ has a unique subgroup $\mathbb{Z}_d$ for each $d \mid n$, we have
\[ \alpha_{[\mathbb{Z}_d]} (\mathbb{Z}_n ; A) = \left\vert \left\{ O \in \mathcal{O}(\mathbb{Z}_n, A) : \vert O \vert = \frac{n}{d} \right\} \right\vert. \]
This number may be determined by ``counting necklaces'', and we shall discuss how to do this in the next section (see Lemma \ref{le:particular}).
\end{example}

\begin{example} \label{ex:klein}
Let $G = \mathbb{Z}_2 \times \mathbb{Z}_2$ be the Klein four-group and $A= \{ 0, 1\}$. As $G$ is abelian, $[H] = \{ H\}$, for all $H \leq G$. The subgroups of $G$ are 
\[ H_1 = G, \ H_2 = \langle (1,0) \rangle, \ H_3 = \langle (0,1) \rangle, \ H_4 = \langle (1,1) \rangle, \ \text{and} \ H_5 = \langle (0,0) \rangle, \]
where $\langle (a,b) \rangle$ denotes the subgroup generated by $(a,b) \in G$. Any configuration $x : G \to A$ may be written as a $2 \times 2$ matrix $(x_{i,j})$ where $x_{i,j} := (i-1,j-1)x$, $i,j \in \{1,2 \}$. The $G$-orbits on $A^G$ are
\begin{align*}
& O_1  := \left\{  \left( \begin{tabular}{cc}
$0$ \ & \ $0$ \\
 $0$ \  & \ $0$
 \end{tabular} \right) \right\}, \  \ O_2 := \left\{ \left( \begin{tabular}{cc}
$1$ \ & \ $1$ \\
 $1$ \  & \ $1$
 \end{tabular} \right) \right\}, \ \ O_3  :=  \left\{  \left( \begin{tabular}{cc}
$1$ \ & \ $0$ \\
 $1$ \  & \ $0$
 \end{tabular} \right), \left( \begin{tabular}{cc}
$0$ \ & \ $1$ \\
 $0$ \  & \ $1$
 \end{tabular} \right) \right\}, \\[.5em]
 & O_4  :=  \left\{ \left( \begin{tabular}{cc}
$1$ \ & \ $1$ \\
 $0$ \  & \ $0$
 \end{tabular} \right), \left( \begin{tabular}{cc}
$0$ \ & \ $0$ \\
 $1$ \  & \ $1$
 \end{tabular} \right)  \right\}, \  \ O_5  :=  \left\{ \left( \begin{tabular}{cc}
$1$ \ & \ $0$ \\
 $0$ \  & \ $1$
 \end{tabular} \right), \left( \begin{tabular}{cc}
$0$ \ & \ $1$ \\
 $1$ \  & \ $0$
 \end{tabular} \right)   \right\} \\[.5em]
& O_6  :=  \left\{ \left( \begin{tabular}{cc}
$1$ \ & \ $0$ \\
 $0$ \  & \ $0$
 \end{tabular} \right), \left( \begin{tabular}{cc}
$0$ \ & \ $1$ \\
 $0$ \  & \ $0$
 \end{tabular} \right), \left( \begin{tabular}{cc}
$0$ \ & \ $0$ \\
 $0$ \  & \ $1$
 \end{tabular} \right), \left( \begin{tabular}{cc}
$0$ \ & \ $0$ \\
 $1$ \  & \ $0$
 \end{tabular} \right) \right\},  \\[.5em] 
& O_7  :=  \left\{ \left( \begin{tabular}{cc}
$0$ \ & \ $1$ \\
 $1$ \  & \ $1$
 \end{tabular} \right), \left( \begin{tabular}{cc}
$1$ \ & \ $0$ \\
 $1$ \  & \ $1$
 \end{tabular} \right), \left( \begin{tabular}{cc}
$1$ \ & \ $1$ \\
 $1$ \  & \ $0$
 \end{tabular} \right), \left( \begin{tabular}{cc}
$1$ \ & \ $1$ \\
 $0$ \  & \ $1$
 \end{tabular} \right) \right\}.
\end{align*}
Hence,
\begin{align*}
& B_{[H_1]}:=O_1 \cup O_2, \ B_{[H_2]}:=O_3, \ B_{[H_3]}:=O_4, \ B_{[H_4]}:=O_5, \ B_{[H_5]}:=O_6 \cup O_7; \\
& \alpha_{[H_i]}(G;A) = 2, \text { for } i \in \{1,5 \}, \text { and } \alpha_{[H_i]}(G;A) = 1, \text{ for } i \in \{ 2,3,4 \}.
\end{align*}
\end{example}

For any $H \leq G$, let $N_G(H) := \{ g \in G : H = g^{-1} H g \} \leq G$ be the \emph{normaliser of $H$ in $G$}. Note that $H$ is always normal in $N_G(H)$. The following result determines the structure of the group $\ICA(G;A)$, and it is a refinement of \cite[Lemma 4]{CRG16a} (c.f. \cite[Theorem 9]{S15} and \cite[Theorem 7.2]{CK15}).

\begin{theorem}[Structure of $\ICA(G;A)$] \label{th:ICA}
Let $G$ be a finite group and $A$ a finite set of size $q \geq 2$. Let $[H_1], \dots, [H_r]$ be the list of all different conjugacy classes of subgroups of $G$. Then,
\[ \ICA(G;A) \cong \prod_{i=1}^r \left( (N_{G}(H_i)/H_i) \wr \Sym_{\alpha_i} \right), \]
where $\alpha_i := \alpha_{[H_i]}(G;A)$.
\end{theorem}
\begin{proof}
Let $B_i := B_{[H_i]}$, and note that all these subshifts are nonempty because of Remark \ref{rk:subgroups}. By Corollary \ref{conjugate}, any $\tau \in \ICA(G;A)$ maps configurations inside $B_i$ to configurations inside $B_i$; hence, $\ICA(G;A)$ is contained in the group 
\[  \prod_{i=1}^r \Sym(B_i) = \Sym(B_1) \times \Sym(B_2) \times \dots \times \Sym(B_r). \]
 For each $1 \leq i \leq r$, fix a $G$-orbit $O_i \subseteq B_i$, and let $\mathcal{O}_i$ be the set of $G$-orbits contained in $B_i$ (so $O_i \in \mathcal{O}_i$). Note that $\mathcal{O}_i$ is a uniform partition of $B_i$ (i.e. all the blocks in the partition have the same size). For any $\tau \in \ICA(G;A)$, Lemma \ref{preserve} implies that the projection of $\tau$ to $\Sym(B_i)$ is contained in the group that preserves this uniform partition, i.e. the projection of $\tau$ is contained in
\[ S(B_i, \mathcal{O}_i ) := \{ \phi \in \Sym(B_i) : \forall P \in \mathcal{O}_i, \  (P)\phi \in \mathcal{O}_i \}.  \]
By \cite[Lemma 2.1(iv)]{AS09}, 
\[ S(B_i , \mathcal{O}_i ) \cong  \Sym(O_i) \wr \Sym_{\alpha_i}.    \]
It is well-known that $\Sym_{\alpha_i}$ is generated by its transpositions. As the invertible cellular automaton constructed in the proof of Lemma \ref{le-ICA} induces a transposition $(xG,yG) \in \Sym_{\alpha_i}$, with $xG, yG \in \mathcal{O}_i$, we deduce that $\Sym_{\alpha_i} \leq \ICA(G;A)$. Now it is clear that the projection of $\ICA(G;A)$ to $\Sym(B_i)$ is exactly $\ICA(O_i) \wr \Sym_{\alpha_i}$.  By \cite[Theorem 4.2A (i)]{DM96}, it follows that
\[  C_{\Sym(O_i)}(G^{O_i})  \cong N_{G}(H_i) / H_i. \]
The result follows because $\ICA(O_i) \cong C_{\Sym(O_i)}(G^{O_i})$.   
\end{proof}

\begin{corollary} \label{cor:structure}
Let $G$ be a finite Dedekind group and $A$ a finite set of size $q\geq 2$. Let $H_1, \dots, H_r$ be the list of different subgroups of $G$. Then,
\[ \ICA(G;A) \cong \prod_{i=1}^r \left( (G/H_i) \wr \Sym_{\alpha_i} \right), \]
where $\alpha_i := \alpha_{[H_i]}(G;A)$.
\end{corollary}
\begin{proof}
As every subgroup of $G$ is normal, the results follows because $[H_i] = \{ H_i \}$ and $N_G(H_i) = G$, for all $1 \leq i \leq r$. 
\end{proof}

\begin{example}
For any $n \geq 2$,
\[ \ICA(\mathbb{Z}_n; A) \cong \prod_{i=1}^r \left( \mathbb{Z}_{d_i} \wr \Sym_{\alpha_i} \right), \]
where $d_1, d_2, \dots, d_r$ are the divisors of $n$, and $\alpha_i : = \alpha_{\mathbb{Z}_{n/d_i}}(\mathbb{Z}_n;A)$.
\end{example}

\begin{example} \label{ex:ICA-klein}
Let $G = \mathbb{Z}_2 \times \mathbb{Z}_2$ and $A= \{ 0, 1\}$. By Example \ref{ex:klein}, 
\[ \ICA(G; A ) \cong (\mathbb{Z}_2)^4 \times (G \wr \Sym_2). \]
\end{example}


\section{Aperiodic Configurations} \label{aperiodic}

In this section, we shall determine the integers $\alpha_{[H]}(G;A)$, for each $H \leq G$, as they play an important role in the decomposition of $\ICA(G;A)$ given by Theorem \ref{th:ICA}. 

The following elementary result links the integers $\alpha_{[H]}(G;A)$ with the sizes of the subshifts $B_{[H]}(G;A)$.

\begin{lemma}
Let $G$ be a finite group of order $n \geq 2$ and $A$ a finite set of size $q \geq 2$. Let $H$ be a subgroup of $G$ of order $m$. Then,
\[ \alpha_{[H]} \cdot n =  m \cdot \vert B_{[H]} \vert . \] 
\end{lemma}
\begin{proof}
The set of $G$-orbits contained in $B_{[H]}$ forms a partition of $B_{[H]}$ into $\alpha_{[H]}$ blocks. The result follows as each one of these $G$-orbits has size $\frac{n}{m}$ by Theorem \ref{orbit-stabiliser}.
\end{proof}

Denote by $[G:H]$ the index of $H$ in $G$ (i.e. the number of cosets of $H$ in $G$). It is well-known that when $G$ is finite, $[G:H] = \frac{\vert G \vert}{\vert H \vert}$. The following result characterises the situations when $B_{H}$ contains a unique $G$-orbit and will be useful in Section \ref{relative-rank}. 

\begin{lemma} \label{alpha1}
Let $G$ be a finite group and $A$ a finite set of size $q\geq 2$. Then, $\alpha_{[H]} (G;A) = 1$ if and only if $[G:H] =2$ and $q = 2$.
\end{lemma}
\begin{proof}
Suppose first that $[G:H] = q = 2$. The subgroup $H$ is normal because it has index $2$, so $Hg = gH$ for every $g \in G$. Fix $s \in G \setminus H$, and define $x \in A^G$ by
\[ (g)x = \begin{cases}
0 & \text{if } g \in H \\
1 & \text{if } g \in sH = Hs.
\end{cases} \]
Clearly, $G_x = H$ and $x \in B_{[H]}$. Let $y \in B_{[H]}$. As $H$ is normal, $[H] = \{ H\}$, so $G_y = H$. For any $h \in H$,
\[ (h)y = (e)y \cdot h^{-1} = (e)y \text{ and } (sh) y = (s) y \cdot h^{-1} = (s)y,  \]
so $y$ is constant on the cosets of $H$. Therefore, either $y = x$, or
\[ (g)y = \begin{cases}
1 & \text{if } g \in H \\
0 & \text{if } g \in sH = Hs.
\end{cases} \]     
In the latter case, $y \cdot s = x$ and $y \in xG$. This shows that there is a unique $G$-orbit contained in $B_{[H]}$, so $\alpha_{[H]}(G;A) = 1$. 

Conversely, suppose that $[G:H] \neq 2$ or $q \geq 3$. If $[G:H]=1$, then $G=H$ and $\alpha_{[H]}(G;A) = q \geq 2$. Now we prove the two cases separately.
\begin{description}
\item[Case 1:] $[G:H] \geq 3$. Define configurations $x_1, x_2 \in A^G$ by
\begin{align*}
(g)x_1 &= \begin{cases} 
1 & \text{if } g \in H, \\
0 & \text{otherwise}.
\end{cases} \\
(g)x_2 & = \begin{cases} 
0 & \text{if } g \in H, \\
1 & \text{otherwise}.
\end{cases}
\end{align*}
It is clear that $x_1, x_2 \in B_{[H]}$ because $G_{x_1} = G_{x_2} = H$. Furthermore, $x_1$ and $x_2$ are not in the same $G$-orbit because the number of preimages of $1$ under $x_1$ and $x_2$ is different (as $[G:H]\geq 3$). Hence, $\alpha_{[H]}(G;A) \geq 2$. 

\item[Case 2:] $q \geq 3$. Define a configuration $x_3 \in B_{[H]}$ by
\[ (g)x_3 = \begin{cases} 
2 & \text{if } g \in H, \\
0 & \text{otherwise},
\end{cases} \]
and consider the configuration $x_1 \in B_{[H]}$ defined in Case 1. Clearly, $x_1$ and $x_3$ are not in the same $G$-orbit because $2 \in A$ is not in the image of $x_1$. Hence, $\alpha_{[H]}(G;A) \geq 2$. 
\end{description}
\end{proof}

For any $H \leq G$, consider the set of \emph{$H$-periodic configurations}:
\[ \Fix(H) := \{ x \in A^G : x \cdot h = x, \forall h \in H \} = \{ x \in A^G : H \leq G_x \}. \] 
By \cite[Corollary 1.3.4.]{CSC10}, we have
\[ \vert \Fix(H) \vert  = q^{[G:H]} . \]

By the Cauchy-Frobenius lemma (\cite[Theorem 1.7A]{DM96}), the total number of $G$-orbits on $A^G$ is
\[ \vert \mathcal{O}(G;A) \vert = \frac{1}{\vert G \vert} \sum_{g \in G} q^{n/ \vert  g \vert},\]
where $\vert  g \vert := \vert \langle g \rangle \vert$ is the \emph{order of the element $g$}. However, we need a somewhat more sophisticated machinery in order to count the number of orbits inside $B_{[H]}(G;A)$.

The \emph{M\"obius function} of a finite poset $(P, \leq)$ is a map $\mu : P \times P \to \mathbb{Z}$ defined inductively by the following equations:
\begin{align*}
\mu(a,a) & = 1, \ \ \forall a \in P, \\
\mu(a,b) &= 0, \ \ \forall a > b, \\
\sum_{a \leq c \leq b} \mu(a,c) &= 0, \ \ \forall a < b.
\end{align*}

 If $L(G)$ is the set of subgroups of $G$, the poset $(L(G), \subseteq)$ is called the \emph{subgroup lattice} of $G$. Let $\mu : L(G) \times L(G) \to \mathbb{Z}$ be the M\"obius function of the subgroup lattice of $G$. In particular, $\mu(H,H) = 1$ for any $H \leq G$, and $\mu(H,K) = -1$ if $H$ is a maximal subgroup of $K \leq G$.

For any $H \le G$ of order $m$, let $p_H$ be the smallest order of a subgroup between $H$ and $G$ (note that $p_H \ge 2 m$):
\[  p_H := \min\{ |K| : H < K \le G \}, \] 
and define 
\[ 	S_H :=  \{ K : H < K \le G, |K| = p_H \}. \]
By convention, $p_G = n$ and $S_G = \emptyset$.

In the next result we use the following asymptotic notation: write $f(n) = o(g(n))$, with $g(n)$ not identically zero, if $\lim_{n \to 0} \frac{f(n)}{g(n)} =0$.

\begin{theorem}\label{le:sizes}
Let $G$ be a finite group of order $n \geq 2$ and $A$ a finite set of size $q \geq 2$. Let $H$ be a subgroup of $G$ of order $m$.
\begin{description}
\item[\textbf{(i)}]
\[ \vert B_{[H]} \vert = \vert [H] \vert  \sum_{K \leq G} \mu(H,K) \cdot q^{n / \vert K \vert }. \]

\item[\textbf{(ii)}] Using asymptotic notation by fixing $G$ and considering $q = |A|$ as a variable, we have
\[ \frac{|B_{[H]}|}{|[H]|}  = q^{n/m} - (|S_H| + o(1)) q^{n/p_H}. \]
\end{description}
\end{theorem}
\begin{proof}
In order to prove part \textbf{(i)}, observe that
\begin{align*}
\vert \Fix(H) \vert = \sum_{H \leq K \leq G} \frac{1}{\vert [K] \vert} \vert B_{[K]} \vert.
\end{align*}
By M\"obius inversion (see \cite[4.1.2]{Ke99}), we obtain that
\[ \vert B_{[H]} \vert = \vert [H] \vert \sum_{K \leq G} \mu(H,K) \cdot \vert \Fix(K) \vert.  \]
The result follows as $\vert \Fix(K) \vert = q^{n / \vert K \vert}$. 

Now we prove part \textbf{(ii)}. The result is clear for $H = G$. Otherwise, we have
\begin{align*}
	\frac{|B_{[H]}|}{|[H]|} &= \sum_{H \le K \le G} \mu(H,K) q^{n/|K|}\\
	&= q^{n/m} - |S_H| q^{n/p_H} + \sum_{K \notin S_H} \mu(H,K) q^{n/|K|}\\
	&= q^{n/m} - (|S_H| - \sum_{K \notin S_H} \mu(H,K) q^{n/|K| - n/p_H} ) q^{n/p_H}.
\end{align*}
The result follows as $q^{n/|K| - n/p_H} = o(1)$. 
\end{proof}

As the M\"obius function of the subgroup lattice of $G$ is not easy to calculate in general, we shall give a few more explicit results by counting the number of so-called \emph{aperiodic configurations}:
\[ \ac(G;A) := \vert \{ x \in A^G : G_x = \{ e \} \} \vert = \vert B_{[\{ e \}]} (G;A) \vert.\]

Part of our motivation to study this number is that, when $H$ is a normal subgroup of $G$, the size of the subshift $B_{[H]}$ is equal to the number of aperiodic configurations with respect to $G/H$.

\begin{lemma}
Let $G$ be a finite group, $A$ a finite set, and $H$ a normal subgroup of $G$. Then,
\[  \vert B_{[H]}(G;A) \vert = \ac( G/H ; A) . \]
\end{lemma}
\begin{proof}
As $H$ is normal, then $G/H$ is a group. By \cite[Proposition 1.3.7.]{CSC10}, there is a $G/H$-equivariant bijection between the configuration space $A^{G/H}$ and $\Fix(H)$. Hence, configurations in $A^{G/H}$ with trivial stabiliser correspond to configurations in $\Fix(H)$ with stabiliser $H$.  
\end{proof}

The following result gives some formulae for the number of aperiodic configurations of various finite groups. 

\begin{lemma}\label{le:particular}
Let $G$ be a finite group of order $n \geq 2$ and $A$ a finite set of size $q \geq 2$.
\begin{description} 
\item[\textbf{(i)}]
\[ \ac(G;A) = \sum_{K \leq G} \mu(\{e \} , K) \cdot q^{n / \vert K \vert }. \]

\item[\textbf{(ii)}] If $G \cong \mathbb{Z}_n$ is a cyclic group, then
\[ \ac(\mathbb{Z}_n, A) =  \sum_{d \mid n} \mu(1,d) \cdot q^{n/d}, \]
where $\mu$ is the classic M\"obius function of the poset $(\mathbb{N}, \mid)$.

\item[\textbf{(iii)}] If $G$ is a $p$-group and $\mathcal{H} := \{ H \leq G : H \text{ is elementary abelian} \}$, then
\[ \ac(G;A) = \sum_{H \in \mathcal{H}} (-1)^{\log_p \vert H \vert} p ^{\binom{\log_p \vert H \vert}{2}} q ^{n/\vert H \vert}  . \]

\item[\textbf{(iv)}] If $G \cong \mathbb{Z}_p^m$ is an elementary abelian group, then 
\[ \ac(\mathbb{Z}_p^m ; A) = q^{p^m} + \sum_{r=1}^{m} (-1)^r q^{p^{m-r}}  p^{r(r-1)/2} \binom{m}{r}_p .  \]
where $\binom{m}{r}_p$ is the Gaussian binomial coefficient:
\[ \binom{m}{r}_p := \frac{(1 - p^m) (1 - p^{m-1}) \dots (1 - p^{m-r+1})}{(1 - p^r)(1 - p^{r-1}) \dots (1 - p )}. \]
\end{description} 
\end{lemma}
\begin{proof}
Part \textbf{(i)} follows by Lemma \ref{le:sizes} \textbf{(iii)}. Part \textbf{(ii)} follows because the subgroup lattice of the cyclic group $\mathbb{Z}_n$ is isomorphic to the lattice of divisors of $n$.

We prove part \textbf{(iii)}. If $G$ is a $p$-group, by \cite[Corollary 3.5.]{HIO89}, for any $H \leq G$ we have
\[ \mu \left( \{e \}, H \right) = \begin{cases}
(-1)^{\log_p \vert H \vert} p^{\binom{\log_p \vert H \vert}{2}} & \text{ if $H$ is elementary abelian,} \\
0 & \text{ otherwise}.
\end{cases} \] 
So the result follows by part \textbf{(i)}. 

Finally, we prove part \textbf{(iv)}. Denote by $\mathcal{H}_r$ the set of elementary abelian subgroups of $G$ of order $p^r$. Then,
\begin{align*}
 \ac(G;A) & = \sum_{r=0}^m (-1)^{r}  \vert \mathcal{H}_r \vert p ^{\binom{r}{2}} q ^{p^m/p^r} \\ 
                & = \sum_{r=0}^m (-1)^{r}  q ^{p^{m-r}} p ^{r(r-1)/2}    \vert \mathcal{H}_r \vert.
\end{align*}
The result follows because the Gaussian binomial coefficient $\binom{m}{r}_p$  gives precisely the number of subgroups of order $p^r$ of $\mathbb{Z}_p^m$ (see \cite[Section 9.2]{C94}).  
\end{proof}

As constant configurations are never aperiodic, the following obvious upper bound of $\ac(G;A)$ is obtained:
\[ \ac(G;A) \leq q^n - q \]
This upper bound is achieved if and only if $n$ is a prime number (so $G$ is a cyclic group). The following lower bound bound of $\ac(G;A)$ was obtained by Gao, Jackson and Seward \cite[Corollary 1.7.2.]{GJS16}:
\[  q^n - q^{n-1} \leq \ac(G;A), \]
which is achieved for small values of $n$ and $q$, as for example, when $n=q=2$, or when $G = \mathbb{Z}_2 \times \mathbb{Z}_2$ and $q = 2$ (see Example \ref{ex:klein}).

For any $d \mid n$, define $G^{(d)} : = \{ g \in G : \vert \langle g \rangle \vert = d \}$. In the next result we improve the known estimates of $\ac(G;A)$. 

\begin{theorem}[Lower bound on apperiodic configurations]\label{low-bound}
Let $G$ be a finite group of order $n \geq 2$ and $A$ a finite set of size $q \geq 2$.  Let $p$ be the smallest prime dividing $n$. Then:
\begin{description}
\item[\textbf{(i)}] $\mathrm{ac}(G; A) = q^n - \left(\frac{ \left\vert G^{(p)} \right\vert}{p-1} + o(1) \right) q^{n/ p}$. 

\item[\textbf{(ii)}] We have the following lower bound: 
\[  q^n - (n-1) q^{n/p} \leq \ac(G; A).  \]
\end{description}
\end{theorem}
\begin{proof}
By Theorem \ref{le:sizes} (ii) with $H =  \{ e \}$, we have
\[ \ac(G;A) = q^{n} - (\vert S_{\{e \}} \vert + o(1) ) q^{n / p_{\{e \}}}. \]
In this case, $p_{\{ e \}}$ is the smallest order of a non-identity element in $G$, which, by Sylow's theorem, is equal to $p$, the smallest prime dividing $n$. Furthermore, $\vert S_{\{ e \}} \vert$, the number of subgroups of $G$ of order $p$, is $\left\vert G^{(p)} \right\vert / (p-1)$, so part (i) follows. 

In order to prove part (ii), let $t_1, \dots, t_r$ be the generators of the minimal subgroups of $G$ (all of which are cyclic groups of prime order). Then, 
\[ 	\ac(G; A) = q^n -  \left\vert \bigcup_{i=1}^r \Fix( \langle t_i \rangle) \right\vert \geq   q^n -   \sum_{i=1}^r  q^{n/ \vert t_i \vert } \geq q^n - (n-1) q^{n/p}. \] 

\end{proof}

The lower bound of Theorem \ref{low-bound} (ii) is much tighter than the one given by \cite[Corollary 1.7.2.]{GJS16}.


\section{Generating Sets of of $\CA(G;A)$} \label{generating}

For a monoid $M$ and a subset $T \subseteq M$, denote by $\langle T \rangle$ the submonoid \emph{generated} by $T$, i.e. the smallest submonoid of $M$ containing $T$. Say that $T$ is a \emph{generating set} of $M$ if $M = \langle T \rangle$; in this case, every element of $M$ is expressible as a word in the elements of $T$ (we use the convention that the empty word is the identity).


\subsection{Memory Sets of Generators of $\CA(G;A)$}

A large part of the classical research on CA is focused on CA with small memory sets. In some cases, such as the elementary Rule 110, or John Conway's Game of Life, these CA are known to be Turing complete. In a striking contrast, when $G$ and $A$ are both finite, CA with small memory sets are insufficient to generate the monoid $\CA(G;A)$.

\begin{theorem}[Minimal memory set of generators of $\CA(G;A)$] \label{minimal-memory}
Let $G$ be a finite group of order $n \geq 2$ and $A$ a finite set of size $q \geq 2$. Let $T$ be a generating set of $\CA(G;A)$. Then, there exists $\tau \in T$ with minimal memory set $S=G$.
\end{theorem}
\begin{proof}
Suppose that $T$ is a generating set of $\CA(G; A)$ such that each of its elements has minimal memory set of size at most $n-1$. Consider the idempotent $\sigma:=(\mathbf{0} \to \mathbf{1}) \in \CA(G; A)$, where $\textbf{0}, \textbf{1} \in A^G$ are different constant configurations. Then, $\sigma = \tau_1 \tau_2 \dots \tau_\ell$, for some $\tau_i \in T$. By the definition of $\sigma$ and Lemma \ref{constant-config}, there must be $1 \leq j \leq \ell$ such that $\left\vert ( A^G_{\text{c}})\tau_j \right\vert = q-1$ and $( A^G_{\text{nc}})\tau_j = A^G_{\text{nc}}$, where 
\[A^G_{\text{c}}:= \{ \mathbf{k} \in A^G : \mathbf{k} \text{ is constant} \} \text{ and } A^G_{\text{nc}} := \{ x \in A^G : x \text{ is non-constant} \}. \]
 Let $S \subseteq G$ and $\mu : A^S \to A$ be the minimal memory set and local function of $\tau := \tau_j$, respectively. By hypothesis, $s := \vert S \vert < n$. Since the restriction of $\tau$ to $A^G_{\text{c}}$ is not a bijection, there exists $\mathbf{k} \in A^G_{\text{c}}$ (defined by $(g)\mathbf{k}:=k \in A$, $\forall g \in G$) such that $\mathbf{k} \not \in ( A^G_{\text{c}})\tau$. 

For any $x \in A^G$, define the $k$-\emph{weight} of $x$ by
\[   \vert x \vert_k := \vert \{ g \in G : (g)x \neq k \} \vert. \] 
Consider the sum of the $k$-weights of all non-constant configurations of $A^G$:
\begin{align*}
 w & := \sum_{x \in A^G_{\text{nc}} } \vert x \vert_k  =  \sum_{x \in A^G} \vert x \vert_k  -  \sum_{x \in A^G_{\text{c}}} \vert x \vert_k   \\
&  =  n(q-1) q^{n-1}  - n(q-1) =  n(q-1) ( q^{n-1} - 1) .  
\end{align*}
In particular, $\frac{w}{n}$ is an integer not divisible by $q$.

For any $x \in A^G$ and $y \in A^S$, define 
\[ \Sub(y, x) := \vert \{ g \in G : y = x \vert_{Sg} \} \vert; \]
this counts the number of times that $y$ appears as a subconfiguration of $x$. Then, for any fixed $y \in A^S$,
\[ N_y  := \sum_{x \in A^G_{\text{nc}}} \Sub(y,x) = \begin{cases}
 n q^{n-s}& \text{if } y \in A^S_{\text{nc}}, \\
 n (q^{n-s} - 1 )& \text{if } y \in A^S_{\text{c}}. 
\end{cases} \]
To see why the previous equality holds, fix $g \in G$ and count the number of configurations $x \in A^G_{\text{nc}}$ such that $y = x \vert_{Sg}$: there are $q^{n-s}$ such configurations if $y$ is non-constant, and $q^{n-s}-1$ if $y$ is constant. The equality follows by counting this for each one of the $n$ elements of $G$.   

Let $\delta : A^2 \to \{0,1 \}$ be the Kronecker's delta function. Since $( A^G_{\text{nc}})\tau = A^G_{\text{nc}}$, we have
\begin{align*}
 w &= \sum_{x \in A^G_{\text{nc}} } \vert (x)\tau \vert_k = \sum_{y \in A^S} N_y  ( 1 - \delta_{(y)\mu, k}  )    \\
& =   n q^{n-s} \sum_{y \in A^S_{\text{nc}}} ( 1 - \delta_{(y)\mu, k} ) +  n ( q^{n-s} - 1 ) \sum_{y \in A^S_{\text{c}}} ( 1 - \delta_{(y)\mu, k} ). 
\end{align*}
Because $\mathbf{k} \not \in ( A^G_{\text{c}})\tau$, we know that $(y)\mu \neq k$ for all $y \in A^S_{\text{c}}$. Therefore,
\[ \frac{w}{n} =  q^{n-s} \sum_{y \in A^S_{\text{nc}}} ( 1 - \delta_{ (y)\mu, k} ) +  ( q^{n-s} - 1 ) q.  \]
As $s <n$, this implies that $\frac{w}{n}$ is an integer divisible by $q$, which is a contradiction. 
\end{proof}


\subsection{Relative Rank of $\ICA(G;A)$ in $\CA(G;A)$} \label{relative-rank}

One of the fundamental problems in the study of a finite monoid $M$ is the determination of the cardinality of a smallest generating subset of $M$; this is called the \emph{rank} of $M$ and denoted by $\Rank(M)$:
\[ \Rank(M) := \min \{ \vert T \vert :  T \subseteq M \text{ and } \langle T \rangle = M \}. \]
It is well-known that, if $X$ is any finite set, the rank of the full transformation monoid $\Tran(X)$ is $3$, while the rank of the symmetric group $\Sym(X)$ is $2$ (see \cite[Ch.~3]{GM09}). Ranks of various finite monoids have been determined in the literature before (e.g. see \cite{ABJS14,AS09,GH87,G14,HM90}). 

In \cite{CRG16a}, the rank of $\CA(\mathbb{Z}_n; A)$, where $\mathbb{Z}_n$ is the cyclic group of order $n$, was studied and determined when $n \in \{ p, 2^k, 2^k p : k \geq 1, \ p \text{ odd prime} \}$. Moreover, the following problem was proposed:
\begin{problem}\label{problem}
For any finite group $G$ and finite set $A$, determine $\Rank(\CA(G;A))$.
\end{problem}                      

For any finite monoid $M$ and $U \subseteq M$, the \emph{relative rank} of $U$ in $M$, denoted by $\Rank(M:U)$, is the minimum cardinality of a subset $V \subseteq M$ such that $\langle U \cup V \rangle = M$. For example, for any finite set $X$, 
\[ \Rank(\Tran(X): \Sym(X)) = 1, \]
as any $\tau \in \Tran(X)$ with $\vert (X) \tau \vert = \vert X \vert -1$ satisfies $\langle \Sym(X) \cup \{ \tau \} \rangle = \Tran(X)$. One of the main tools used to determine $\Rank(\CA(G;A))$ is based on the following result (see \cite[Lemma 3.1]{AS09}).

\begin{lemma} \label{le:preliminar}
Let $M$ be a finite monoid and let $U$ be its group of units. Then,
\[ \Rank(M) = \Rank( M : U ) + \Rank(U). \]
\end{lemma}

We shall determine the relative rank of $\ICA(G;A)$ in $\CA(G;A)$ for any finite abelian group $G$ and finite set $A$. In order to achieve this, we prove two lemmas that hold even when $G$ is nonabelian and have relevance in their own right. 

\begin{lemma} \label{le:action-orbit}
 Let $G$ be a finite group and $A$ a finite set of size $q\geq 2$. Let $\tau \in \CA(G;A)$ and $x\in A^G$. If $(x)\tau \in xG$, then $\tau \vert_{xG} \in \Sym(xG)$. 
\end{lemma} 
\begin{proof}
It is enough to show that $(xG)\tau = xG$ as this implies that $\tau \vert_{xG} : xG \to xG$ is surjective, so it is bijective by the finiteness of $xG$. Since $(x)\tau \in xG$, we know that $(x)\tau G = xG$. Hence, by Lemma \ref{preserve}, $(xG)\tau = (x) \tau G = x G$.  
\end{proof}

\begin{remark}\label{GoE}
Recall that a \emph{Garden of Eden} (GoE) of $\tau \in \CA(G;A)$ is a configuration $x \in A^G$ such that $x \notin (A^G)\tau$. As $A^G$ is finite in our setting, note that $\tau$ is non-invertible if and only if it has a GoE. Moreover, by $G$-equivariance, $x$ is a GoE of $\tau$ if and only if $xG$ is a GoE of $\tau$, so we shall talk about GoE $G$-orbits, rather than GoE configurations. 
\end{remark}

Denote by $\mathcal{C}_G$ the set of conjugacy classes of subgroups of $G$. For any $[H_1], [H_2] \in \mathcal{C}_G$, write $[H_1] \leq [H_2]$ if $H_1 \leq g^{-1} H_2 g$, for some $g \in G$.

\begin{remark}
The relation $\leq$ defined above is a well-defined partial order on $\mathcal{C}_G$. Clearly, $\leq$ is reflexive and transitive. In order to show antisymmetry, suppose that $[H_1] \leq [H_2]$ and $[H_2] \leq [H_1]$. Then, $H_1 \leq g^{-1} H_2 g$ and $H_2 \leq f^{-1} H_1 f$, for some $f,g \in G$, which implies that $\vert H_1 \vert \leq \vert H_2 \vert$ and $\vert H_2 \vert \leq \vert H_1 \vert$. As $H_1$ and $H_2$ are finite, $\vert H_1 \vert = \vert H_2 \vert$, and $H_1 = g^{-1} H_2 g$. This shows that $[H_1] = [H_2]$.     
\end{remark}

\begin{lemma}\label{le:idem}
Let $G$ be a finite group and $A$ a finite set of size $q\geq 2$. Let $x , y \in A^G$ be such that $xG \neq yG$. There exists a non-invertible $\tau \in \CA(G;A)$ such that $(x)\tau = y$ if and only if $G_x \leq G_y$.
\end{lemma}
\begin{proof}
In general, for any $\tau \in \CA(G;A)$ such that $(x)\tau = y$, we have $G_x \leq G_y$, because we may argue as in the first line of the proof of Lemma \ref{le-ICA}.

Conversely, suppose that $G_x \leq G_y$. We define an idempotent transformation $\tau_{x,y} : A^G \to A^G$ as follows:
\[ (z) \tau_{x,y} := \begin{cases}
y \cdot g  & \text{if } z = x \cdot g, \\ 
z & \text{otherwise},
 \end{cases}  \quad \quad \forall z \in A^G. \] 
Note that $\tau_{x,y}$ is well-defined because $x \cdot g = x \cdot  h$, implies that $gh^{-1} \in G_x \leq G_y$, so $y \cdot g = y \cdot h$. Clearly, $\tau_{x,y}$ is non-invertible and $G$-equivariant, so $\tau_{x,y}\in \CA(G;A) \setminus \ICA(G;A)$. 
\end{proof}

\begin{corollary} 
Let $G$ be a finite group and $A$ a finite set of size $q\geq 2$. Let $x , y \in A^G$ be such that $xG \neq yG$. There exists a non-invertible $\tau \in \CA(G;A)$ such that $(xG)\tau = yG$ if and only if $[G_x] \leq [G_y]$.
\end{corollary}

Consider the directed graph $(\mathcal{C}_G, \mathcal{E}_G)$ with vertex set $\mathcal{C}_G$ and edge set
\[ \mathcal{E}_G := \left\{ ([H_i], [H_j]) \in \mathcal{C}_G^2 : [H_i] \leq [H_j] \right\}.  \]
When $G$ is abelian, this graph coincides with the subgroup lattice of $G$.

\begin{remark}
Lemma \ref{le:idem} may be restated in terms of $(\mathcal{C}_G, \mathcal{E}_G)$. By Lemma \ref{le:action-orbit}, loops $([H_i],[H_i])$ do not have corresponding non-invertible CA when $\alpha_{[H_i]}(G;A)=1$.  
\end{remark}

Recall that an action of $G$ on a set $X$ is \emph{transitive} if for any $x,y \in X$ there exists $g \in G$ such that $x \cdot g = y$ (i.e. $X = xG$, for any $x \in X$). The following result will be useful in order to prove the main theorem of this section.

\begin{lemma} \label{Dedekind}
Let $G$ be a finite group and $A$ a finite set of size $q\geq 2$. Then $\ICA(G;A)$ is transitive on every $G$-orbit on $A^G$ if and only if $G$ is a finite Dedekind group. 
\end{lemma}
\begin{proof}
Let $x \in A^G$ be a configuration. By Theorem \ref{th:ICA}, the group $\ICA(G;A)$ acts on $xG$ as the group $N_G(G_x)/G_x$ via the action $x \cdot (G_x g) := x \cdot g$, for any $G_x g \in N_G(G_x)/G_x$. Note that $N_G(G_x)/G_x$ is transitive on $xG$ if and only if $G = N_G(G_x)$, which holds if and only if $G_x$ is normal in $G$. As any subgroup of $G$ occurs as a stabiliser of a configuration, this shows that $\ICA(G;A)$ is transitive on $xG$, for all $x \in A^G$, if and only if every subgroup of $G$ is normal. 
\end{proof}

\begin{remark}
It is obvious, by definition, that $G$ is transitive on a $G$-orbit $xG$. However, Lemma \ref{Dedekind} establishes a criterion for the transitivity of the group $\ICA(G;A)$ on $xG$.
\end{remark}

\begin{theorem}[Relative rank of $\ICA(G;A)$ on $\CA(G;A)$] \label{th:relative rank}
Let $G$ be a finite group and $A$ a finite set of size $q\geq 2$. Let $I_2(G)$ be the set of subgroups of $G$ of index $2$: 
\[ I_2(G) = \{ H \leq G : [G:H] = 2 \}. \]
Then,
\[ \Rank(\CA(G;A):\ICA(G;A)) \geq  \begin{cases}
\vert \mathcal{E}_G \vert -  \vert I_2(G) \vert & \text{if } q=2, \\
\vert \mathcal{E}_G \vert & \text{otherwise},
\end{cases} \]
with equality if and only if $G$ is a finite Dedekind group.
\end{theorem}
\begin{proof}
Let $[H_1], [H_2], \dots, [H_r]$ be the list of different conjugacy classes subgroups of $G$ with $H_1 = G$. Suppose further that this is ordered such that 
\begin{equation} \label{order}
\vert H_1 \vert \geq \vert H_2 \vert \geq \dots \geq \vert H_r \vert 
\end{equation}
For each $1 \leq i \leq r$, let $\alpha_i := \alpha_{[H_i]}(G;A)$ and $B_i := B_{[H_i]}(G;A)$. Fix orbits $x_i G \subseteq B_i$ such that $G_{x_i} <  G_{x_j}$ whenever $[H_i] < [H_j]$. For every $\alpha_i \geq 2$, fix orbits $y_i G \subseteq B_i$ such that $x_i G \neq y_i G$. Consider the set  
\[ V := \left\{ \tau_{x_i,x_j} : G_{x_i} < G_{x_j} \right\} \cup \left\{ \tau_{x_i, y_i} : \alpha_i \geq 2 \right\}, \]	
and $\tau_{x_i, x_j}$ and $\tau_{x_i, y_i}$ are the idempotents that map $x_i$ to $x_j$ and $x_i$ to $y_i$, respectively, as defined in Lemma \ref{le:idem}. Observe that
\[ \vert V \vert = \vert \mathcal{E}_G \vert - \sum_{i=1}^n \delta(\alpha_i, 1) =  \begin{cases}
\vert \mathcal{E}_G \vert -  \vert I_2(G) \vert & \text{if } q=2, \\
\vert \mathcal{E}_G \vert & \text{otherwise},
\end{cases} \] 
where the last equality follows by Lemma \ref{alpha1}. 

\begin{claim}
The relative rank of $\ICA(G;A)$ on $\CA(G;A)$ is at least $\vert V \vert$. 
\end{claim}
\begin{proof}
Suppose there exists $W \subseteq \CA(G;A)$ such that $\vert W \vert < \vert V \vert$ and 
\[ \left\langle \ICA(G ; A) \cup W  \right\rangle = \CA(G ; A). \]

Let $\tau \in \CA(G;A)$. We say that $\tau$ is a \emph{Unique-Garden-of-Eden CA (UCA) of type} $(B_i, B_j)$ if the following holds:
\begin{description}
\item[$(\star)$] $\tau$ has a unique Garden of Eden $G$-orbit $O_i$ (i.e. $(A^G)\tau =A^G \setminus O_i$), and it satisfies $O_i \subseteq B_i$ and $(O_i)\tau \subseteq B_j$.
\end{description}
Note that UCA of type $(B_i, B_i)$ only exist when there are at least two different orbits in $B_i$, i.e. $\alpha_i \geq 2$.

For example, the idempotents $\tau_{x_i, y_i} \in V$, with $\alpha_i \geq 2$, are UCA of type $(B_i,B_i)$, while the idempotents $\tau_{x_i, x_j}$, with  $G_{x_i} < G_{x_j}$, are UCA of type $(B_i,B_j)$ with $O_i = x_i G$. Note that $\tau \in \CA(G;A)$ is a UCA of type $(B_i, B_j)$ if and only if $\langle \ICA(G; A) , \tau \rangle$ contains a UCA of type $(B_i,B_j)$ if and only if all non-invertible elements of $\langle \ICA(G; A) , \tau \rangle$ are UCA of type $(B_i, B_j)$ (because any $\phi \in \ICA(G;A)$ always satisfies $\phi(B_k) = B_k$ for all $k$, by Corollary \ref{conjugate}).  

As $\vert W \vert < \vert V \vert$, and $V$ has exactly one UCA of each possible type (see Lemma \ref{le:idem}), there must be $\tau \in V$ such that there is no UCA in $W$ of the same type as $\tau$. Without loss of generality, suppose that the type of $\tau$ is $(B_i, B_j)$ (possibly with $i=j$). We finish the proof of the claim by showing that there is no UCA of type $(B_i, B_j)$ in $\langle W \rangle$. This would imply that there is no UCA of type $(B_i, B_j)$ in $\left\langle \ICA(G ; A) \cup W  \right\rangle$, contradicting that $\left\langle \ICA(G ; A) \cup W  \right\rangle = \CA(G ; A)$.

Assume that
\begin{equation} \label{factorisation}
 \omega := \omega_1 \dots \omega_s \dots \omega_\ell \in \langle W \rangle,  \text{ with } \omega_m \in W,\ \forall m = 1, \dots, \ell,  
\end{equation}
is a UCA of type $(B_i,B_j)$. First note that, as $\omega$ has no GoE in $B_k$, for all $k \neq i$, then $(B_k) \omega = B_k$. Hence,
\begin{equation} \label{help}
(B_k) \omega_m = B_k \text{ for all } k \neq i \text{ and } m = 1 \dots \ell,
\end{equation}
because $\omega_m$ cannot map any $G$-orbit of $B_k$ to a different subshift $B_c$, as there is no CA mapping back $B_c$ to $B_k$ (see Lemmas \ref{le:idem} and \ref{le-ICA}), and $\omega_m$ does not have GoE inside $B_k$ because this would be a GoE for $\omega$ inside $B_k$. 

Now, observe that each non-invertible $\omega_m$ that appears in (\ref{factorisation}) has a unique GoE orbit $\Omega_m$ (inside $B_i$ because of (\ref{help})). This is true because if $\Omega_m$ consists of more than one orbit, then $(A^G) \omega_m = A^G \setminus \Omega_m$ implies that the size of $(A^G) \omega$ is strictly smaller than $\vert A^G \vert - \vert O_i \vert$, where $O_i \subseteq B_i$ is the unique GoE $G$-orbit of $\omega$.  

Let $\omega_s$ the the first non-invertible CA that appears in (\ref{factorisation}). We finish the proof by showing that $(\Omega_s) \omega_s \subseteq B_j$. Let $\Omega_s^\prime = (\Omega_s)\omega_{s-1}^{-1} \dots \omega_1^{-1}$. There are three possibilities: 
\begin{description}
\item[\textbf{Case 1:}] $(\Omega_s)\omega_s = P_c \subseteq B_c$ for $c \neq i,j$. Then:
\[ (\Omega_s^\prime) \omega = (\Omega_s)\omega_s \dots \omega_\ell = (P_c) \omega_{s+1} \dots \omega_\ell \subseteq B_c,  \]
where the last contention is because of (\ref{help}). However, as $\omega$ maps all orbits of $B_i$ to $B_i \cup B_j$, this case is impossible.

\item[\textbf{Case 2:}] $(\Omega_s)\omega_s  \subseteq B_i$. If $i = j$, then $\omega_s$ is a UCA of the same type as $\omega$. Hence, let $i \neq j$. By the uniqueness of $\Omega_s$, $(B_i \setminus \Omega_s) \omega_s =(B_i \setminus \Omega_s)$, so there exists a $G$-orbit $Q \subseteq B_i$, $Q \neq \Omega_s$, such that $(\Omega_s) \omega_s = (Q) \omega_s$. Let $Q^\prime = (Q)\omega_{s-1}^{-1} \dots \omega_1^{-1}$. Then,
\[ (\Omega_s^\prime) \omega = (\Omega_s) \omega_s \dots \Omega_\ell  = (Q) \omega_s \dots \omega_\ell = (Q^\prime) \omega.  \]
However, as $\omega$ maps its only GoE orbit to $B_j$, it does not collapse orbits in $B_i$. So this case, with $i \neq j$ is impossible.

\item[\textbf{Case 3:}] $(\Omega_s)\omega_s \subseteq B_j$. In this case, $\omega_s$ is a UCA of type $(B_i, B_j)$. 
\end{description}
 
In any case, we obtain a contradiction with the assumption that $W$ has no UCA of type $(B_i, B_j)$. Therefore, $\langle W \rangle$ has no UCA of type $(B_i, B_j)$.  
\end{proof}

\begin{claim}
If $G$ is a Dedekind group, then $\Rank(\CA(G;A):\ICA(G;A)) = \vert V \vert$.
\end{claim}
\begin{proof}
We will show that  
\[ \CA(G;A) = M:= \left\langle \ICA(G;A) \cup V \right\rangle. \]
For any $\sigma \in \CA(G;A)$, consider $\sigma_i \in \CA(G;A)$, $1 \leq i \leq r$, defined by
\[ (x)\sigma_i  = \begin{cases}
(x)\sigma & \text{if } x \in B_i \\
x & \text{otherwise}.
\end{cases}\]
By Lemmas \ref{le-ICA} and \ref{le:idem}, we know that $G_x \leq G_{(x)\sigma}$, for all $x \in A^G$, so $(B_i)\sigma \subseteq \bigcup_{j \leq i} B_j$ for all $i$ (recall the order given by (\ref{order})). Hence, we have the decomposition
\[ \sigma = \sigma_1 \circ \sigma_2 \circ \dots \circ \sigma_r.  \] 

We shall prove that $\sigma_i \in M$ for all $1 \leq i \leq r$. For each $\sigma_i$, decompose $B_i = B_i^{\prime} \cup B_i^{\prime \prime}$, where
\begin{align*}
B_i^\prime &:= \bigcup \left\{ P \in \mathcal{O}(G;A) : P \subseteq B_i \text{ and } (P)\sigma_i \subseteq B_j \text{ for some } j <i \right\}, \\
B_i^{\prime \prime} &:= \bigcup \left\{ P \in \mathcal{O}(G;A) : P \subseteq B_i \text{ and } (P)\sigma_i \subseteq B_i \right\}.
\end{align*}
If $\sigma^\prime_i$ and $\sigma^{\prime \prime}_i$ are the CA that act as $\sigma_i$ on $B_i^\prime$ and $B_i^{\prime \prime}$, respectively, and fix everything else, then $\sigma_i = \sigma_i^{\prime} \circ \sigma_{i}^{\prime \prime}$. We shall prove that $\sigma_i^\prime \in M$ and $\sigma_i^{\prime \prime} \in M$.
\begin{enumerate}
\item We show that $\sigma_i^{\prime} \in M$. For any orbit $P \subseteq B_i^\prime$, the orbit $Q:= (P)\sigma_{i}^{\prime}$ is contained in $B_j$ for some $j < i$. By Theorem \ref{th:ICA}, there exists an involution 
\[ \phi \in \left( (N_G(G_{x_i})/G_{x_i}) \wr \Sym_{\alpha_i} \right) \times \left( (N_G(G_{x_j})/G_{x_j}) \wr \Sym_{\alpha_j} \right) \leq \ICA(G;A) \]
that induces the double transposition $(x_i G, P) (x_j G,  Q)$. By Lemma \ref{Dedekind}, $\ICA(G;A)$ is transitive on $x_iG$ and $x_jG$ (as $G$ is Dedekind), so we may take $\phi$ such that $(x_i) \phi \sigma_i^\prime = (x_j) \phi$. Then,
\[  (z) \sigma_i^\prime  = (z) \phi  \tau_{x_i, x_j} \phi , \ \forall z \in P = (x_i G)\phi. \]
As $\sigma_i^\prime$ may be decomposed as a product of CA that only move one orbit in $B_i^{\prime}$, this shows that $\sigma_i^\prime \in M$.

\item We show that $\sigma_{i}^{\prime \prime} \in M$. In this case, $\sigma_i^{\prime \prime} \in \Tran(B_i)$. In fact, as $\sigma_i^{\prime \prime}$ preserves the partition of $B_i$ into $G$-orbits, Lemma \ref{le:action-orbit} and \cite[Lemma 2.1 (i)]{AS09} imply that $\sigma_i^{\prime \prime} \in (G/G_{x_i}) \wr \Tran_{\alpha_i}$. If $\alpha_i \geq 2$, the monoid $\Tran_{\alpha_i}$ is generated by $\Sym_{\alpha_i} \leq \ICA(G;A)$ together with the idempotent $\tau_{x_i, y_i}$. Hence, $\sigma_i^{\prime \prime} \in M$.    
\end{enumerate}
Therefore, we have established that $\CA(G;A) = \left\langle \ICA(G;A) \cup V \right\rangle$. 
\end{proof}

\begin{claim}
If $G$ is not a Dedekind group, then $\Rank(\CA(G;A):\ICA(G;A)) > \vert V \vert$.
\end{claim}
\begin{proof}
As $G$ is not Dedekind, so there is a subgroup $H \leq G$ which is not normal. Hence, $H = G_{x_i}$ for a non-constant configuration $x_i \in A^G$, and, by the proof of Lemma \ref{Dedekind}, $\ICA(G;A)$ is not transitive on $x_i G$. Consider the idempotent $\tau_{i,1} \in V$. Then $\tau_{i,1} =  (P \to x_1)$, where $P$ is the $\ICA(G;A)$-orbit inside $x_i G$ that contains $x_i$ and $x_1 \in A^G$ is a constant configuration. Let $Q$ be an $\ICA(G;A)$-orbit inside $x_i G$ such that $Q \neq P$. As there is no $\phi \in \ICA(G;A)$ mapping $P$ to $Q$, we have $(Q \to x_i) \notin \langle \ICA(G;A) \cup V \rangle$. Therefore, the edge $([H],[ G ])$ of the graph on $\mathcal{C}_G$ must be counted at least twice for its contribution towards the relative rank of $\ICA(G;A)$ on $\CA(G;A)$. The claim follows. 
\end{proof}
 
\end{proof}

Using Theorem \ref{th:relative rank}, we may find an upper bound for the smallest size of a generating set of $\CA(G;A)$, when $G$ is a finite Dedekind group.

\begin{corollary} \label{cor:bound}
Let $G$ be a finite Dedekind group and $A$ a finite set of size $q \geq 2$. Suppose that $\Rank(G) = m$ and that $G$ has $r$ different subgroups. Then,
\[ \Rank(\CA(G;A))  \leq  m (r -1) + \frac{1}{2} r (r + 5). \] 
\end{corollary}
\begin{proof}
Observe that for any $\alpha \geq 3$ and any $H \leq G$, we have $\Rank((G/H)\wr \Sym_{\alpha}) \leq m + 2$ because
\[ \{ ((g_1, e, \dots, e); \id), \dots, ((g_m,e, \dots, e); \id) , ((e,\dots,e); (1,2)), ((e,\dots,e);(1,2,\dots, \alpha))  \} \]
is a generating set of $(G/H)\wr \Sym_{\alpha}$, whenever $\{ g_1, \dots, g_m \}$ is a generating set for $G$.

Let $H_1, H_2, \dots, H_r$ be the list of different subgroups of $G$ with $H_1 = G$.  For each $1 \leq i \leq r$, let $\alpha_i := \alpha_{[H_i]}(G;A)$.
Thus, by Lemma \ref{le:preliminar}, Corollary \ref{cor:structure}, and Theorem \ref{th:relative rank} we have:
\begin{align*}
\Rank(\CA(G;A)) &= \Rank(\ICA(G;A)) +  \Rank(\CA(G;A):\ICA(G;A)) \\
& \leq  \sum_{i=1}^r \Rank((G/H_i)\wr \Sym_{\alpha_i}) + \vert \mathcal{E}_G \vert \\ 
& \leq \Rank(\Sym_q) + \sum_{i=2}^r (m + 2) + \binom{r}{2} + r \\
& \leq 2 + (r-1)(m+2) + \frac{1}{2}r(r-1) + r \\
& = m (r -1) + \frac{1}{2} r (r + 5).
\end{align*}

\end{proof}

\begin{figure}[h]
\centering
\begin{tikzpicture}[vertex/.style={circle, draw, fill=none, inner sep=0.48cm}]
    \vertex{1}{1}{2}    \node at (1,3.4) {$H_5 \cong \mathbb{Z}_1$};  
    \vertex{2}{3}{1}    \node at (3,1.7) {$H_4 \cong \mathbb{Z}_2$};   
   \vertex{3}{1}{1}    \node at (1,1.7) {$H_3 \cong \mathbb{Z}_2$};  
   \vertex{4}{-1}{1}    \node at (-1,1.7) {$H_2 \cong \mathbb{Z}_2$};  
   \vertex{5}{1}{0}   \node at (1,0) {$H_1 = G$};   

  \arc{1}{2}
  \arc{1}{3}
 \arc{1}{4}
\arc{2}{5}
\arc{3}{5}
\arc{4}{5}
   \end{tikzpicture} 
\caption{Hasse diagram of subgroup lattice of $G = \mathbb{Z}_2 \times \mathbb{Z}_2$.}
\label{Fig1}
   \end{figure}
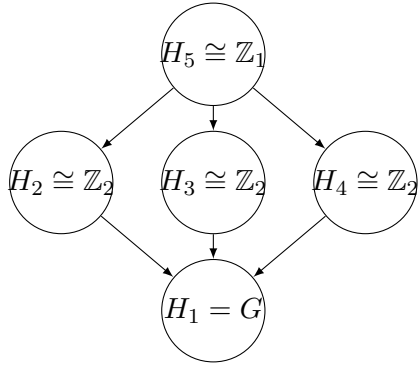  

The bound of Corollary \ref{cor:bound} may become tighter if we actually know $\vert \mathcal{E}_G \vert $ and $\Rank(G/H_i)$, for all $H_i \leq G$.

\begin{example}
Let $G=\mathbb{Z}_2 \times \mathbb{Z}_2$ be the Klein-four group and $A = \{ 0,1 \}$. With the notation of Example \ref{ex:klein}, Figure \ref{Fig1} illustrates the Hasse diagram of the subgroup lattice of $G$ (i.e. the actual lattice of subgroups is the transitive and reflexive closure of this graph). Hence, by Theorem \ref{th:relative rank} and Example \ref{ex:ICA-klein},
\begin{align*}
& \Rank(\CA(G;A):\ICA(G;A))  =  \vert \mathcal{E}_G \vert  - 3 = 12 - 3 = 9,\\
& \Rank(\CA(G;A))   \leq 9 + 9 = 18, \text{ as } \Rank(\ICA(G;A)) \leq 9.  
\end{align*}
\end{example}


\section{Conclusions}

In this paper we studied the monoid $\CA(G;A)$ of all cellular automata over a finite group $G$ and a finite set $A$. Our main results are the following:
\begin{enumerate}
\item We completely determined the structure of the group of invertible cellular automata $\ICA(G;A)$ in terms of the structure of $G$ (Theorem \ref{th:ICA}).

\item We improved the known lower bound on the number of aperiodic configurations of $A^G$ (Theorem \ref{low-bound}).

\item We showed that any generating set of $\CA(G;A)$ must have at least one cellular automaton whose minimal memory set is $G$ itself (Theorem \ref{minimal-memory}).

\item We gave a lower bound for the minimal size of a set $V \subseteq \CA(G;A)$ such that $\ICA(G;A) \cup V$ generates $\CA(G;A)$, and we showed that this lower bound is achieved if and only if all subgroups of $G$ are normal (Theorem \ref{th:relative rank}).
\end{enumerate}

Most of our results are particularly good for finite Dedekind groups, i.e. finite groups in which all subgroups are normal (this includes all finite abelian groups). 

Some open problems and directions for future work are the following:

\begin{enumerate}
\item Examine further the case when $G$ is not a Dedekind group; in particular, determine the relative rank of $\ICA(G;A)$ on $\CA(G;A)$.

\item Improve the upper bound on $\Rank(\CA(G;A))$ given by Corollary \ref{cor:bound}.

\item Study generating sets of $\CA(G;A)$ when $G$ is an infinite group.

\item Study other algebraic properties of the monoid $\CA(G;A)$.
\end{enumerate}

\section{Acknowledgments}
This work was supported by the EPSRC grant EP/K033956/1. We kindly thank Turlough Neary and Matthew Cook for their invitation to submit this paper and for the organisation of the conference AUTOMATA 2016. We also thank the referees of this paper for their insightful suggestions.

\Addresses

\end{document}